\documentclass{article}
\usepackage[utf8]{inputenc}
\usepackage{amsmath,amssymb}
\usepackage{tikz-cd}
\usepackage{url}

\usepackage{hyperref}
\usepackage{cleveref}
\usepackage{bibnames}

\usepackage{enumitem}
\usepackage{amsthm}
\newtheorem{theorem}{Theorem}[section]

\newtheorem{defi}[theorem]{Definition}

\newcommand{\rel}{\texttt{Rel}}
\newcommand{\cpx}{\texttt{Cpx}}

\newcommand{\cech}{\v{C}ech }

\newcommand{\tr}{\text{S}}
\newcommand{\bl}{E}
\newcommand{\id}{\text{Id}}

\newcommand{\R}{(R,X,Y)}
\newcommand{\Ro}{(R_0,X_0,Y_0)}
\newcommand{\Ri}{(R_1,X_1,Y_1)}

\usepackage[british]{babel}
\usepackage[useregional]{datetime2}
\DTMlangsetup[en-GB]{showdayofmonth=false}

\tikzset{
    labl/.style={anchor=south, rotate=270, inner sep=.5mm}
}

\usepackage{authblk}

\title{The Rectangle Complex of a Relation}
\author[1]{Morten Brun}
\author[1]{Lars M. Salbu}
\date{\today}
\affil[1]{Department of Mathematics, University of Bergen.}

\begin{document}
%\date{\today}

%\address{%
%Allégaten 41\\
%Realfagbygget\\
%5007 Bergen\\
%Norway}
%----------Author 1
%\author[Brun]{Morten Brun}
%\address{%
%Department of Mathematics\\
%University of Bergen\\
%Bergen, Norway}
%\email{Morten.Brun@uib.no}

%----------Author 3
%\author[Salbu]{Lars M. Salbu}
%\address{%
%Department of Mathematics\\
%University of Bergen\\
%Bergen, Norway}
%\email{Lars.Salbu@uib.no}

%----------classification, keywords, date
%\subjclass{55N31}

%\keywords{TDA, Relations, Dowker's Theorem, Functorial, FCA}

%\date{\today}
%----------additions

%\maketitle
\maketitle
\begin{abstract}

We construct a simplicial complex, the \textit{rectangle complex} of a relation $R$, and show that it is homotopy equivalent to the 
\textit{Dowker complex} of \(R\). 
This results in a short and conceptual proof of functorial versions of Dowker's Theorem used in 
topological data analysis. 

\end{abstract}

\section{Introduction}
%A relation is a subset of a product of sets. 

In his 1952 paper ``Homology Groups of Relations''\cite{Dowker}, C.H. Dowker associated a simplicial complex, the \textit{Dowker complex \(D(R)\)}, to every relation, that is, to every subset \(R\) of the product \(X \times Y\) of some sets \(X\) and \(Y\).

In this paper, we define the \textit{rectangle complex \(\bl(R)\)} of a relation $R$ to be the simplicial complex with vertex set \(R\) whose simplices are all subsets of non-empty \textit{rectangles} in \(R\), where by a rectangle we mean a subset of \(R\) of the form \(U \times V\) for \(U \subseteq X\) and \(V \subseteq Y\). 
The projection \(X \times Y \to X\) restricts to a map \(\pi_R \colon R \to X\), and this map induces a map \(\pi_R \colon \bl(R) \to D(R)\) of simplicial complexes.
Our Theorem \ref{RectangleThm} states that \(\pi_R \colon \bl(R) \to D(R)\) is a homotopy
equivalence. This result is an easy consequence of Quillen's Theorem A \cite[Prop. 1.6]{Quillen} and the Borsuk Nerve Theorem \cite[Cor. 9.3]{Borsuk}. 

Dowker's Theorem \cite[Thm. 10.9]{bjorner} states that the Dowker complex of the relation \(R\) is homotopy equivalent to the Dowker complex of the transpose relation \(R^T \subseteq Y \times X\) consisting of pairs \((y,x)\) with \((x,y) \in R\). %(Dowker originally, in \cite[Thm. 1]{Dowker}, gave an isomorphism on homology groups). 
A functorial version of Dowker's Theorem is a direct consequence of Theorem \ref{RectangleThm}. 
To our best knowledge, this is a new proof of Dowker's Theorem. 
Dowker originally showed that $D(R)$ and $D(R^T)$ have isomorphic homology groups \cite[Thm. 1]{Dowker}, which was strengthened by Bj\"orner \cite[Thm. 10.9]{bjorner}.% when he used the Nerve theorem to show that $D(R)$ and $D(R^T)$ are homotopy equivalent. 

On a historical note, the work of Dowker was picked up by Atkin in the 60s in the field of social science under the name Q-theory \cite{atkin}. Since the 80s, relations are studied in the field of \emph{formal concept analysis} (FCA).
A relation \(R\) (called the \emph{formal context}) is studied through its maximal rectangles (called the \emph{formal concepts}). The formal concepts form a lattice (called the \emph{concept lattice}) which is studied in great detail (see e.g. Ganter and Obiedkov \cite{FCAbook} for an introduction to FCA). 
Several papers relate the work of Dowker to FCA. For example Freund et al. \cite{music} use the concept lattice to find a strong deformation retract of the Dowker complex, and Ayzenberg \cite{ayzenberg} shows that the classifying space of the concept lattice (removing its initial and terminal point) is homotopy equivalent to the Dowker complex. Ayzenberg also gives an elegant proof of the Nerve Theorem. With the rise of topological data analysis (TDA) in the 2000s, Dowker's work has received increased attention, since the Dowker complex generalizes the nerve of a cover, and the \cech and Rips complexes are instances of nerves of covers. This has been elaborated on by Blaser and Brun \cite{blaser}, and Virk \cite{Virk_2021}. Chowdhury and Mémoli \cite[Thm. 3]{memoli} stated a functorial version of Dowker's Theorem that applies in the study of filtered simplicial complexes, and the functorial Dowker theorem has since been generalized by Virk \cite{Virk_2021}.

In Section \ref{sec: rel and dow} we formally define relations and their Dowker complex. In Section \ref{sec:rec} we introduce the rectangle complex. In Section \ref{sec: fiber} we present a homotopy equivalence from the rectangle complex of a relation to its Dowker complex. In Section \ref{sec:Dowkertheorem} we prove a strong version of Dowker's Theorem. Finally, in Section \ref{sec:conc} we give concluding remarks.

\section{Relations and the Dowker Complexes}\label{sec: rel and dow}
In this section we give the prerequisite definitions of relations and simplicial complexes. We also define the Dowker complex, which is a way of constructing simplicial complexes from relations.
\begin{defi}
    A \textbf{relation} is a triple of sets $(R,X,Y)$ where $R\subseteq X\times Y$.
\end{defi}
\begin{defi}
    A \textbf{morphism of relations} $f:\Ro\to\Ri$ is a pair of functions $f=(f_1:X_0\to X_1,f_2:Y_0\to Y_1)$ such that $(x,y)\in R_0$ implies $(f_1(x),f_2(y))\in R_1$.
\end{defi}
If we take pair-wise composition of morphisms, we inherit associativity and identities from the category of sets, in particular we have $\id_{(R,X,Y)}=(\id_X,\id_Y)$. We write $\rel$ for this category of relations.\footnote{Note that the category of relations refers sometimes in literature to the category where the objects are sets and the morphisms are relations.}

\begin{defi}
    A \textbf{simplicial complex} is a pair $(K,V)$ consisting of a \textbf{vertex set} $V$ together with a set $K$ of finite subsets of $V$ closed under inclusions, i.e. if $\sigma$ is in $K$ and $\tau\subseteq\sigma$, then $\tau$ is also in $K$.
\end{defi}
An element $\sigma\in K$ is called a \textbf{simplex}. A \textbf{simplicial map} $F$ from $(K,V)$ to $(K',V')$ is a function on vertex sets $F:V\to V'$ such that $\sigma \in K$ implies $F(\sigma):=\{F(s)\,|\,s\in\sigma\}\in K'$. Composition of simplicial maps is given as the composition of functions. We denote the category of simplicial complexes by $\cpx$. Each simplicial complex $(K,V)$ has an associated topological space $|K|$ called its \textbf{geometric realization} \cite[3.1.14]{Spanier}. 

\begin{defi}[{\cite[p. 85]{Dowker}}]
    The \textbf{Dowker complex} of a relation $(R,X,Y)$ is the simplicial complex $(D(R),X)$ where
\begin{equation*}%\label{dowkerdef}
    D(R)=\{\sigma\subseteq X\,|\,\exists y\in Y \text{ with } \sigma \times \{y\}\subseteq R\}.
\end{equation*}
\end{defi}
We say that the simplex $\sigma\in D(R)$ is \textbf{witnessed} by $y\in Y$ if $\sigma \times \{y\}\subseteq R$. Note that $D:\rel\to\cpx$ is a functor, taking a morphisms $f=(f_1,f_2):\Ro\to\Ri$ of relations to the simplicial map $D(f):D(R_0)\to D(R_1)$ induced by the function $f_1:X\to X'$ of vertex sets.

Dowker's Theorem compares the Dowker complex of a relation $(R,X,Y)$ to the Dowker complex of the transpose relation $(R^T,Y,X)$ defined as follows:
\begin{defi}
    Given a relation $(R,X,Y)$, the \textbf{transpose relation} $(R^T,Y,X)$ is the relation where $(y,x)\in R^T$ if and only if $(x,y)\in R$.
\end{defi}

\section{The Rectangle Complex}\label{sec:rec}
In this section we introduce basic building blocks of relations, and use them to construct a simplicial complex. We start with some general notions. 
\begin{defi}
    A \textbf{rectangle} of a relation $\R$ is a finite subset of \(R\) of the form \(U \times V\) with \(U\subseteq X\) and \(V\subseteq Y\).
\end{defi}
A rectangle $U\times V$ in $R$ is a \textbf{formal concept} if it is inclusion-maximal, namely if $U\times V\subsetneq U'\times V'$ then $U'\times V'\not\subseteq R$. Formal concepts are the main objects of interest in the field of formal concept analysis (see \cite{FCAbook} for an introduction). An algorithm for computing formal concepts has been found by Norris \cite{Norris} among others. 
For a relation $\R$ the \textbf{first-coordinate projection} is the function $\pi_R: R\to X$, where $\pi_R(x,y)=x$. Similarly, the \textbf{second-coordinate projection} $\widehat{\pi}_R: R\to X$ sends $(x,y)$ to $y$.
%Looking at the set of rectangles and taking the inclusion-closure we get a simplicial complex.
\begin{defi}[The Rectangle Complex]
    Let \(\R\) be a relation with coordinate projections \(\pi_R \colon R \to X\) and \(\widehat \pi_R \colon R \to Y\).
    The \textbf{rectangle complex} of $\R$ is the simplicial complex $(\bl(R),R)$ 
    with vertex set \(R\) and simplices given by the subsets of $R$ contained in a rectangle of $R$.
    That is,
    \begin{equation*}
        \bl(R) = \{\tau\subseteq X\times Y\,|\, \pi_R(\tau) \times \widehat \pi_R(\tau) \subseteq R\}.
    \end{equation*}
\end{defi}
Note that $\bl:\rel\to\cpx$ is a functor taking a morphism $f=(f_1,f_2):\Ro\to\Ri$ of relations to the simplicial map $E(f):E(R_0)\to E(R_1)$ with $f|_{R_0}:R_0\to R_1$ as vertex map.

For a relation $\R$ the first-coordinate projection $\pi_R: R\to X$, defines a simplicial map
\begin{equation}\label{pimap}
    \pi_R: \bl(R)\to D(R).
\end{equation}
To see that this is indeed simplicial, we note for $\tau\in E(R)$ the simplex $\pi_R(\tau)$ is witnessed by every element $y\in \widehat{\pi}_R(\tau)$, and is thus a simplex in the Dowker complex. For every morphism of relation $f=(f_1,f_2):\Ro\to\Ri$ the diagram 
\begin{equation*}%\label{nattrans}
    \begin{tikzcd}
        \bl(R_0) 
        \arrow[r,"\pi_{R_0}"] 
        \arrow[d,"E(f)"] 
        &
        D(R_0)\arrow[d,"D(f)"] 
        \\
        \bl(R_1)
        \arrow[r,"\pi_{R_1}"] 
        &
        D(R_1)
    \end{tikzcd}
\end{equation*}
commutes. In particular, the collection of all $\pi_R$ defines a natural transformation $\pi:\bl\to D$. In the next section we continue to show that $\pi_R$ is a homotopy equivalence for all relations $R$.

Figure \ref{Fig: eksempel} shows the two Dowker complexes \(D(R)\) and \(D(R^T)\) and the rectangle complex $E(R)$ of the relation \(\R\) with \(X = \{a,b,c,d\}\), \(Y = \{1,2,3,4\}\) and \(R = \{(a,2),(a,4),(b,1),(b,2), (c,1), (c,4), (d,1), (d,3)\}\).

\begin{figure}[!h]
\centering
\includegraphics[width=0.9\textwidth]{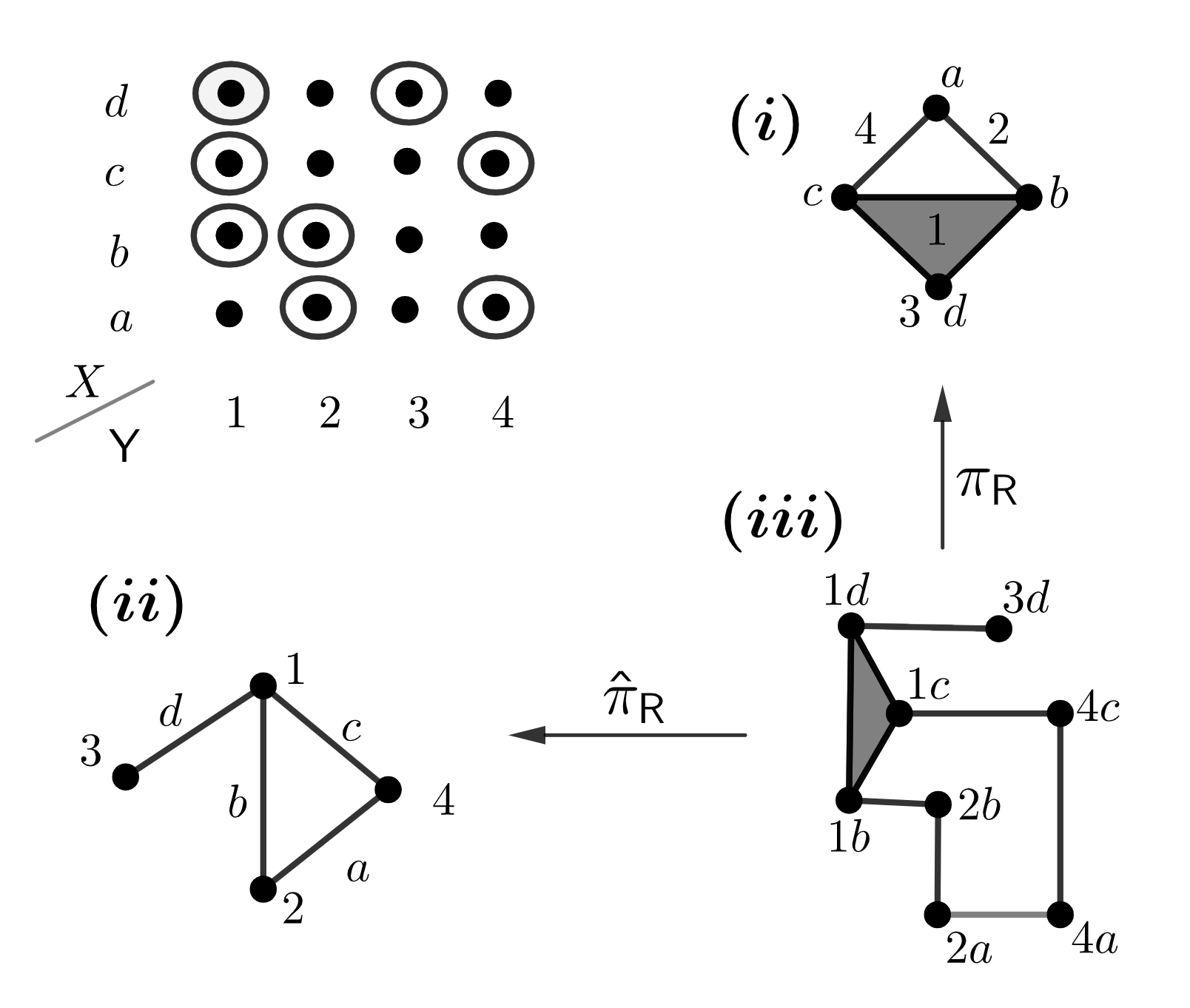}
\caption{\label{Fig: eksempel} The relation $R$ is given by the circled points in $X\times Y$, where $X=\{a,b,c,d\}$ and $Y=\{1,2,3,4\}$. Here, $(i)$ is the Dowker complex $D(R)$, $(ii)$ is the Dowker complex $D(R^T)$ and $(iii)$ is the rectangle complex $E(R)$ whose vertex set is $R$ (we write $1d$ for $(1,d)$, etc.). Note how the hole $\{1c,4c,4a,2a,2b,1b\}$ maps onto the hole $\{c,a,b\}$ by $\pi_R$ and onto the hole $\{1,4,2\}$ by $\widehat{\pi}_R$, hinting that the maps $\pi_R$ and $\widehat{\pi}_R$ are homotopy equivalences.}
\end{figure}

\section{The Rectangle Complex and the Dowker Complex are Homotopy Equivalent}\label{sec: fiber}
If $K$ is a simplicial complex, then $K_\subseteq$ is a partially ordered set, and its order complex is the barycentric subdivision of $K$. It is well-known that the barycentric subdivision of $K$ is homeomorphic to $K$ itself \cite[3.3.9]{Spanier}.

For an order-preserving map $f:P_\leq\to Q_\leq$, we define its \textbf{fiber} at $q\in Q_\leq$ to be the partially ordered subset
\begin{equation*}%\label{fiberdef}
    f/q = \{p\in P_\leq\,|\, f(p)\leq q\}.
\end{equation*}

There is a famous result from Quillen \cite[Prop. 1.6]{Quillen}, saying that if the fiber $f/q$ is contractible\footnote{meaning the geometric realization of the order complex of the fiber is contractible} for every point $q\in Q_\leq$, then $f$ gives a homotopy equivalence of order complexes. We state the special case for simplicial complexes: 

\begin{theorem}[{\cite[Thm. 10.5]{bjorner}, \cite[Prop. 1.6]{Quillen}}]\label{thmquillen}
Let $F:(K,V)\to (K',V')$ be a simplicial map. If the fiber $F/\sigma=\{\tau\in K\,|\, F(\tau)\subseteq \sigma\}$ is contractible for every $\sigma\in K'$, then $F$ induces a homotopy equivalence $|F|:|K|\to |K'|$.
\end{theorem}

The nerve of a covering \(\mathcal U\) is the simplicial complex with vertex set \(\mathcal U\) and simplices given by subsets of \(\mathcal U\) with non-empty intersection.

\begin{theorem}[{Borsuk Nerve Theorem \cite[Cor. 9.3]{Borsuk}}]\label{borsuk}
Let \(\mathcal U\) be a cover of a simplicial complex \(K\) by simplicial subcomplexes. If every finite intersection of elements in \(\mathcal U\) is contractible, then the geometric realization of the nerve of \(\mathcal U\) is homotopy equivalent to the geometric realization of \(K\)
\end{theorem}

Let us remark that the Borsuk Nerve Theorem \ref{borsuk} is a conseqence Theorem \ref{thmquillen} as shown by Bj\"orner in the proof of \cite[Thm. 10.6]{bjorner}. We use these results to show that the map in \eqref{pimap} is a homotopy equivalence.

\begin{theorem}\label{thmnew}\label{RectangleThm}
Let $R\subseteq X\times Y$ be a relation. The simplicial map $\pi_R:\bl(R)\to D(R)$ where $\pi_R(x,y)=x$ induces a homotopy equivalence $|\pi_R|:|\bl(R)|\to|D(R)|$. 
\end{theorem}
\begin{proof}
If $\sigma$ is a simplex in $D(R)$, then the inverse image $\pi^{-1}_R(\sigma)$ consists of all subsets of rectangles of $R$ of the form $\sigma\times B$. 
Note that if we have two subsets $B,B'\subseteq Y$, then $\sigma\times (B\cup B')\subseteq R$ if and only if both $\sigma\times B\subseteq R$ and $\sigma\times B'\subseteq R$. 
So the inverse image $\pi^{-1}_R(\sigma)$ is the set of all subsets of the set $\sigma\times Y(\sigma)$, where $Y(\sigma) = (\bigcup_{\sigma\times B\subseteq R} B)$. In particular $\pi^{-1}_R(\sigma)$ is a simplex.

The fiber $\pi_R/\sigma$ is the union
of the inverse images $\pi_R^{-1}(\tau)$ for \(\tau\) a subset of \(\sigma\). 
We consider the cover \(\mathcal U\) of the fiber $\pi_R/\sigma$ given by the simplices 
$\pi_R^{-1}(\tau)$ for \(\tau\subseteq\sigma\). Since an intersection of simplices is either a simplex or empty, the Nerve Theorem \ref{borsuk} implies that the geometric realization of the fiber $\pi_R/\sigma$ is homotopy equivalent to the nerve of \(\mathcal U\).
However, as we explain below, the vertex $\pi^{-1}_R(\sigma)$ is a cone point of the nerve of \(\mathcal U\), and thus the nerve is contractible. In order to see that the vertex $\pi^{-1}_R(\sigma)$ is a cone point of the nerve of \(\mathcal U\), let $\gamma$ be a simplex in the nerve of \(\mathcal U\). Then \(\gamma\) is a collection of sets of the form \(\tau \times Y(\tau)\) whose intersection is non-empty. Observe that  \(\tau \subseteq \tau'\) implies \(Y(\tau) \supseteq Y(\tau')\). If \((x,y)\) is a point contained in every element \(\tau \times Y(\tau)\) of \(\gamma\), and if \(y' \in Y(\sigma)\), then the point \((x,y')\) is a point in $\pi^{-1}_R(\sigma)$ contained in every element of \(\gamma\). This means that \(\gamma \cup \{\pi^{-1}_R(\sigma)\}\) is a simplex in the nerve of \(\mathcal U\), and thus the vertex \(\pi^{-1}_R(\sigma)\) is a cone point of this cover. 
The result now follows from Theorem \ref{thmquillen}.
\end{proof}

\section{Dowker's Theorem}\label{sec:Dowkertheorem}

In this section we use Theorem \ref{thmnew} to prove a strong version of Dowker's Theorem.

Recall that \(R^T = \{(y,x) \, \vert \, (x,y) \in R\}\), so $U\times V\subseteq R$ if and only if $V\times U\subseteq R^T$. Thus, the transpose map $\tr:X\times Y\to Y\times X$ defined by $\tr(x,y)=(y,x)$ gives an isomorphism on simplicial complexes $\tr_R:\bl(R)\to\bl(R^T)$ when restricted to $R$. In particular, $|\tr_R|$ is a homeomorphism and the second-coordinate projection $\widehat{\pi}_R=\pi_{R^T}\circ S_R:E(R)\to D(R^T)$ is a homotopy equivalence. We arrive at the following result:

\begin{theorem}
\label{simpdow}
    For any relation $\R$, the maps $|\pi_R|$ and $|\widehat{\pi}_R|$ are homotopy equivalences. For any morphism of relations $f=(f_1,f_2):\Ro\to\Ri$ the diagram 
    \begin{equation}\label{diagramsimpdow}
    \begin{tikzcd}
     D(R_0)\arrow[d,"D(f)"]&
     \bl(R_0)\arrow[l,swap,"\pi_{R_0}"]\arrow[r,"\widehat{\pi}_{R_0}"]\arrow[d,"E(f)"]&
     D(R_0^T)\arrow[d,"D(f^T)"]\\
     D(R_1)&
     \bl(R_1)\arrow[l,swap,"\pi_{R_1}"]\arrow[r,"\widehat{\pi}_{R_1}"]&
     D(R_1^T)\
    \end{tikzcd}
    \end{equation}
    commutes. \qed
\end{theorem}

Continuing, if we let $\phi_R:|D(R)|\to|\bl(R)|$ be a chosen homotopy inverse of $|\pi_R|$. The composition 
\begin{equation}\label{psidef}
    |D(R)|\xrightarrow[\simeq]{\phi_R}|\bl(R)|\xrightarrow[\cong]{|\tr_R|}|\bl(R^T)|\xrightarrow[\simeq]{|\pi_{R^T}|}|D(R^T)|,
\end{equation}
which we denote by $\Psi_R$, is also a homotopy equivalence. In particular, the Dowker complexes $|D(R)|$ and $|D(R^T)|$ are homotopy equivalent. %Which was observed by Dowker in 1952 \cite[Thm. 1]{Dowker}. 

Taking the geometric realization of diagram \eqref{diagramsimpdow} and replacing $|\pi_{R_i}|$ for $i=0,1$ by its homotopy inverse $\phi_{R_i}$ we lose strict commutativity, but the diagram still commutes up to homotopy.

\begin{theorem}[{\cite[Thm. 5.2]{Virk_2021}}, general case of {\cite[Thm. 3]{memoli}}]\label{funcDowk}
    For any relation $\R$, the map $\Psi_R$ is a homotopy equivalence. 
    For any morphism of relations $f=(f_1,f_2):\Ro\to\Ri$ the diagram
    \begin{equation*}
    \begin{tikzcd}
        \vline D(R_0)\vline
        \arrow[r,swap,"\Psi_{R_0}"] 
        \arrow[r,"\simeq"] 
        \arrow[d,"\vline D(f)\vline"] 
        &
        \vline D(R_0^T)\vline \arrow[d,"\vline D(f^T)\vline"] 
        \\
        \vline D(R_1) \vline
        \arrow[r,swap,"\Psi_{R_1}"]
        \arrow[r,"\simeq"] 
        &
        \vline D(R_1^T)\vline
    \end{tikzcd}
    \end{equation*}
    commutes up to homotopy.\qed
\end{theorem}
This result is often called the \textbf{functorial Dowker theorem}, a term was coined in \cite[Thm. 3]{memoli}, where the case when $f_1$ and $f_2$ are inclusions is considered. In Theorem \ref{funcDowk} we arrive at the more general case stated in \cite[Thm. 5.2]{Virk_2021} in terms of covers and nerves.

\section{Conclusion}\label{sec:conc}
We have introduced the rectangle complex of a relation and used it to give a short proof of Dowker's Theorem. An advantage of this proof is that all constructions are functorial, so we get the general functorial Dowker theorem (Theorem \ref{funcDowk}) without extra work. We obtain strict commutativity on the level of simplicial complexes as stated in Theorem \ref{simpdow}. The notion of (maximal) rectangles of a relation has already had great success in formal concept analysis, and we demonstrate that it is also of interest in topology.

\bibliographystyle{plainurl} %acm
\bibliography{biblio}

\end{document}